\documentclass[12pt]{amsart}
\usepackage{amsmath,amssymb, amscd, graphicx}
\usepackage[all]{xy}

\makeatletter
\@addtoreset{equation}{section}
\makeatother
\usepackage{enumerate}

\usepackage{color}
\usepackage{ulem}
\usepackage{comment}

\makeatletter
\@addtoreset{equation}{section}
\makeatother
\usepackage{enumerate}

\def\ee{{\mathrm{e}}}

\def\RR{{\mathbb R}}
\def\EE{{\mathbb E}}

\def\cC{{\mathcal{C}}}

\def\cK{{\mathcal{K}}}
\def\fM{{\mathfrak{M}}}

\def\SE{{\mathrm{SE}}}
\def\SO{{\mathrm{SO}}}
\def\se{{\mathfrak{se}}}
\def\so{{\mathfrak{so}}}
\def\fg{{\mathfrak{g}}}

\def\coker{{\mathrm{coker}}}
\def\img{{\mathrm{img}}}

\newtheorem{theorem}{Theorem}[section]
\newtheorem{definition}[theorem]{Definition}

\newtheorem{proposition}[theorem]{Proposition}
\newtheorem{corollary}[theorem]{Corollary}
\newtheorem{remark}[theorem]{Remark}
\newtheorem{lemma}[theorem]{Lemma}

\def\book#1{\rm{#1}, }
\def\paper#1{\textit{#1}, }
\def\jour#1{\rm{#1}, }
\def\yr#1{({\rm{#1}) }}
\def\vol#1{\textbf{#1}}
\def\pages#1{\rm{#1}}
\def\page#1{\rm{#1}}

\def\publaddr#1{\rm{#1}, }
\def\publ#1{\rm{#1}, }
\def\by#1{{\rm{#1}, }}

\begin{document}

\title{Mathematics in Caging of Robotics} 

\author{
  Hiroyasu Hamada,
  Satoshi Makita,
  Shigeki Matsutani
}

\maketitle

\date{\today}

\begin{abstract}
It is a crucial problem in robotics field
to cage an object using robots like multifingered hand.
However the problem what is the caging
for general geometrical objects and robots
has not been well-described in mathematics
though there were many rigorous studies on the methods how to
cage an object by certain robots.
In this article, we investigate the caging problem 
more mathematically
and describe the problem in terms of recursion of
the simple euclidean moves. Using the description,
we show that the caging has
the degree of difficulty 
which is closely related to a combinatorial
problem and a wire puzzle.   
It implies that in order to capture an object by
caging,
from a practical viewpoint the difficulty plays an important
role.  
\end{abstract}


{\bf{keywords}}
caging, euclidean move, wire puzzle, 
robotics

{\bf{AMS:}}
51M04, 
57N15, 
57N35, 
70B15, 
70E60, 

\section{Introduction}

In robotics fields, caging is a type of grasping where robots capture an
object by surrounding or hooking it.
Thus the caging problem based on the shape of the robots and the object
is addressed on 
geometrical representation.
Its mathematical description has been partially studied with
focusing on methodology how to cage an object by certain robots.
Though it is rigorous, it is not for 
arbitrary target objects and robots.
In this article, we propose an essential of caging to
describe arbitrary target objects and robots from a mathematical viewpoint,
and then it naturally leads us to a degree of difficulty of
escaping and caging. It is a novel concept of the caging which
is connected with practical approaches. 

Caging or holding an object has been discussed in mathematics field such
as \cite{C,Z}, and has been applied to robotic manipulation in parallel.
Rimon and Blake raised a caging by two circular robots driven by one
parameter in two dimensional planar space \cite{RB}, and formulated
its conditions.
Wang and Kumar
proposed caging by multi-robot cooperation with mathematical abstract
formulas \cite{WK}. 
More than three dimensional caging problem is formulated by \cite{PS},
although only circular and spherical robots are referred.
There studies discuss existence of object's free movable space closed by
the robots.
Hence path connectivity of the free space for the captured object
\cite{RMF} is an important matter to investigate whether the object is
caged or able to escape.
As mentioned above, confinement of caging formation is studied
by previous works such as \cite{RB, PS, RMF}, particularly for problems
of two dimensional caging.
Additionally although path connectivity can be examined by using
probabilistic search algorithm \cite{DS,MP}, 
the difficulty of caging constraint is not quantitatively qualified.

Therefore this paper aims to describe caging problem in robotics in the
following viewpoint:
 to apply the formulation to arbitrary robots and objects;
 to reveal difficulty of caging constraint mathematically.

From a mathematical viewpoint, we recall the fundamental
fact that a compact connected $n$-dimensional
topological manifold $M$ in $n$-dimensional
euclidean space $\EE^n$ could be wrapped by the $(n-1)$-dimensional
sphere $S^{n-1}$, which corresponds
to holding in the real world.
This fact is described well by the homotopy group
$\pi_{n-1}$. It is based on the mathematical fact that $S^{n-1}$
in $\EE^n$
divides $\EE^n$ into its inner side and outer side. 
The subspace $M$ has the configuration space $[M]$ by the euclidean
moves $\psi$, i.e., $[M]=\{\psi(M)\}$.
Let us fix the sphere $S^{n-1}$ and then 
it divides $[M]$ to the part 
 $[M]_{in}$ of 
the inner side  
of $S^{n-1}$, that of the outer side, $[M]_{out}$, and 
the intersection part $[M]_{int}$, i.e., 
$[M]=[M]_{in}\coprod
[M]_{out}\coprod
[M]_{int}$ and
$[M]_{int}:= \{M' \in [M]\ | \ M'\cap S^{n-1}\neq \emptyset\}$.
If $[M]_{in}$ is not empty, it means 
that $S^{n-1}$ holds $M \in [M]_{in}$.

On the other hand,
 caging is to divide the configurations space $[M]$ 
using an $(n-2)$ dimensional topological manifold $\cK$.
Then there arises a problem
how $\cK$ with 2-codimension
can divide the configuration space $[M]$ mathematically.
This is a fundamental problem on caging.

There are many mathematical studies on this problem
to cage some proper geometrical objects using special $\cK$
as mentioned above.
Further
Fruchard and Zamfirescu \cite{F,Z}
geometrically studied the geometrical conditions 
whether a circle holds a convex object $[M]$.
Rodriguez, Mason and Ferry investigated 
geometrical and topological properties the multifinger cages rigorously
\cite{RMF}.  

However there was no investigation on the
above fundamental problem for general $M$ and $\cK$.
Though the path connectivity of an object $M$ in $\EE^n \setminus \cK$
 studied well in  \cite{DS,MP,RMF} is an essential property of 
caging, 
it is not sufficient.
For example if we regard a wire puzzle constituting two pieces as
a pair of 
a robot $\cK$ and a target object $M$, $M$ is not caged by $\cK$
because there is a path between $M$ {\lq\lq}in{\rq\rq} of $\cK$
and that of the outer side of $\cK$. However a wire puzzle should
be considered as an example of caging and holding from a practical
viewpoint. 
It is related to the probabilistic treatment of the path connectivity
\cite{DS,MP}. However the path space is
very complicate in general \cite{BT} and it is very difficult to
assign a probabilistic measure in the path space.

In order to introduce a degree of escaping
mathematically,
in this article, we describe the fundamental problem more mathematically,
and then we show that caging is represented in terms of recursion of
the simple euclidean moves,  i.e., the piecewise euclidean move
defined in Definitions \ref{def:PWEM} and \cite{def:ell-redu}.
The move means
that caging is classified countably and naturally leads us to 
the degree of difficulty of escaping of $[M]$ from $\cK$.
It is closely related to a combinatorial explosion and
the wire puzzle.
It means that there might be
a difference between
practical caging and complete caging.
When we capture a complicate object by
caging, we propose that the difficulty 
should be proactively considered 
from a practical viewpoint.
Further if we treat the euclidean moves probabilistically, we could
assign a natural measure on the moves.

\section{Mathematical Preliminary}

Let us consider the $n$-dimensional real euclidean space
$\EE^n$ and the $n$-dimensional euclidean group 
$\SE(n):=\RR^n\rtimes \SO(n)$.
The element $g$ of $\SE(n)$ acts on each point $x$ of $\EE^n$
by $\Phi_g(x) \in \EE^n$ such that 
for $g, g' \in \SE(n)$, 
$\Phi_{g' g}(x) = \Phi_{g'} \Phi_{g}(x)$ and
for the identity element $e \in \SE(n)$,
$\Phi_{e}(x) = id (x) =x$ \cite{AM,G}.
The action $\Phi_g$ of $g\in\SE(n)$ on $\EE^n$ has
the matrix representation:
for $g$, there are element $u\in \RR^n$ and $A\in\SO(n)$ such that
for $x\in \EE^n$, 
\begin{displaymath}
\Phi_g(x) = 
\begin{pmatrix}
1 & 0 \\
u & A \end{pmatrix}
\begin{pmatrix}1 \\ x
\end{pmatrix}
=
\begin{pmatrix}1 \\
 A x +u
\end{pmatrix}.
\end{displaymath}

In this article, let us refer a $\ell$-dimensional
topological manifold in $\EE^n$, {\it{$\ell$-subspace}}.
We say that the $\ell$-subspaces $M$ and $M'$ of $\EE^n$
are {\it{congruent}} if there
is an element $g$ of $\SE(n)$ such that $\Phi_g(M)=M'$.
We denote it by $M \simeq M'$.
Let $[M]$ be the configuration space of $M$;
$[M]:=\{\Phi_g(M)\ | \ g \in \SE(n)\}$.
The quotient space of the set of the $n$-subspaces 
divided by $\SE(n)$
is denoted by $\fM$, which classifies the shape of subspace in
$\EE^n$. We call $\fM$ 
{\it{moduli of the shapes}}; $[M] \in \fM$.

We consider a continuous map $\psi :[0,1] \to \SE(n)$,
i.e., $\psi \in \cC^0([0,1],\SE(n))$ with a fixing point
$\psi(0)=id$,
where $\cC^0(N,F)$ means the set of $F$-valued continuous
functions over $N$.

For given $\psi \in \cC^0([0,1],\SE(n))$, 
the action $\Phi_{\psi}$ on $\EE^n$
parameterized by $t\in [0,1]$
is called 
{\it{orbit by $\psi$}}.

\begin{lemma}
For an $n$-subspace $M \subset \EE^n$, 
the action $\Phi_{\psi(t)} (M)$ induces a congruent 
family $\{\Phi_{\psi(t)}(M)\}_{t\in [0,1]}$ 
\end{lemma}

For the Lie group $\SE(n)$,
there is its Lie algebra $\se(n)$ whose element
$\fg$ satisfies $\exp(\fg) \in \SE(n)$;
for the economy of notations, we use the 
same notation $\fg$ for its matrix representation. 
\begin{lemma} \label{lm:LieAlg}
For $\displaystyle{
\fg=
\begin{pmatrix}
0 & 0 \\
\xi & \omega \end{pmatrix} \in \se(n),
}$
where $\omega \in \so(n)$ and $\xi \in  \RR^n$,
\begin{displaymath}
\exp(\fg)=
\ee^{\fg}=
\begin{pmatrix}
1 & 0 \\
v(\omega)
\xi  & \ee^\omega \end{pmatrix},
\end{displaymath}
where 
$v(\omega):=
\displaystyle{
\sum_{k=0}^\infty 
\frac{1}{(k+1)!}\omega^{k} = 
\int^1_0 \ee^{s\omega} d s}$,
\begin{displaymath}
\omega v(\omega) = \ee^\omega - I_n,
\end{displaymath}
where $I_n$ is the $n \times n$ unit matrix.
If $\omega$ is regular,
$v(\omega)=
(\ee^\omega-I_n)\omega^{-1}$.
\end{lemma}

\begin{proof}
The straightforward computation leads
$\displaystyle{
\fg^k=\begin{pmatrix}
0 & 0 \\
\omega^{k-1}\xi & \omega^k \end{pmatrix}
}$
and thus we have the result.
\end{proof}

Here we recall the properties of $\so(n)$ \cite{Kn}:
\begin{lemma} \label{lm:LieAlg_so(n)}
\begin{enumerate}
\item $\displaystyle{\dim_\RR \so(n) = \frac{n(n-1)}{2}}$,

\item the matrix representation of an element $\omega$ of $\so(n)$ 
is given
by $(\omega_{i j})$ such that ${}^t\omega=-\omega$,
i.e., $\omega_{i j}=-\omega_{j i}$ ($i,j = 1, 2, \ldots, n$),

\item the maximal rank of the matrix representation of $\so(n)$ is
$(n-1)$ if $n=odd$ and $n$ otherwise,
and

\item for the matrix representation
$\omega \in \so(n)$,
we have the natural decomposition,
\begin{displaymath}
 \RR^n = \omega(\RR^n) \oplus \ker\omega
\end{displaymath}
by considering $\omega:\RR^n \to \RR^n$, i.e.
 the cokernel
 $\coker\,\omega$ agrees with the kernel of $\omega$.

\end{enumerate}
\end{lemma} 

\begin{proof}
1, 2, and 3 are obvious as in \cite[p.63]{Kn}.
Using the euclidean inner product,
let us show $\omega(\RR^n)^\perp=\ker \omega$ which is
equivalent with 4.
Let us consider an element $x$ of 
$\omega(\RR^n)^\perp=(\img \omega)^\perp$, which means
$(x, \omega y)=0$ for every $y \in \RR^n$,
and equivalently $(^t \omega x, y)$ vanishes.
The element $x$ must belongs to $\ker {}^t \omega$ 
or $\ker\, \omega$  because of ${}^t\omega= - \omega$. 
\end{proof}

Since for $\fg \in \se(n)$ and $t\in [0,1]$, $t\fg$ belongs to
$\se(n)$, 
by using this relation between $\SE(n)$ and $\se(n)$,
we consider a euclidean move
$\psi(t) = \exp(t \fg)$, which we call
{\it{the simple euclidean move}}. 

\begin{lemma}
For $t\in[0,1]$ and 
$\displaystyle{
\fg=
\begin{pmatrix}
0 & 0 \\
\xi & \omega \end{pmatrix} \in \se(n),
}$
\begin{displaymath}
\exp(t \fg)=
\begin{pmatrix}
1 & 0 \\
v_t(\omega)\xi  & \ee^{t\omega} \end{pmatrix}.
\end{displaymath}
where $v_t(\omega)=t v(t\omega) =
\displaystyle{\int^t_0 \ee^{s\omega} d s}$
satisfying
\begin{displaymath}
\omega v_t(\omega) = \ee^{t\omega} - I_n.
\end{displaymath}
If $\omega$ is regular,
$v_t(\omega)=(\ee^{t\omega}-I_n)\omega^{-1}$.
\end{lemma}

\begin{proof}
$v_t(\omega)=t v(t\omega) =
\displaystyle{\sum_{k=1}^\infty \frac{1}{k!}
t^{k} \omega^{k-1}}$
$\displaystyle{=t\int^1_0 \ee^{st\omega} d s}$.
By replacing $s$ with $st$,
$\displaystyle{v_t(\omega) =
\int^t_0 \ee^{s\omega} d s}$.
\end{proof}

Though the action of the euclidean move
should be regarded as
a {\lq\lq}rotation{\rq\rq} in the projective space
$P\RR^n$ via SL($n,\RR^{n+1}$) in $\RR^{n+1}$,
the simple euclidean move is given by
the following lemma: 

\begin{lemma} \label{lmm:RT}
For $t\in[0,1]$ and 
$\displaystyle{
\fg=
\begin{pmatrix}
0 & 0 \\
\xi & \omega \end{pmatrix} \in \se(n)}$,
the orbit of $x\in \EE^n$ by 
the simple euclidean move $\psi(t):=\ee^{t\fg}$, i.e., 
$
\psi(t)x 
=\ee^{t \omega} x + v_t(\omega) \xi,
$
 is reduced to 
\begin{equation}
 \psi(t)x = \ee^{t \omega} (x + \xi_0) - \xi_0 + \xi_1 t,
\label{eq:se_SE}
\end{equation}
where $\xi =  \omega \xi_0 + \xi_1$
for $\xi_1 \in \ker \omega$ and $\xi_0 \in \RR^n$.
It means the following:
\begin{enumerate}
\item
If $\xi_1$ vanishes, $\psi(t)$ means
a rotation by $\ee^{t \omega}$,
at a center $-\xi_0\in \EE^n$
by identifying $\RR^n$ with $\EE^n$ as a set. 

\item
If $\omega$ vanishes, $\psi(t)$ is the translation
\begin{displaymath}
\psi(t)x = x + \xi t.
\end{displaymath}

\item If both $\omega$ and $\xi_1$ do not vanish,
it is
a mixed move of the rotation and the translation.
\end{enumerate}
\end{lemma}

\begin{proof}
Since this action satisfies
\begin{displaymath}
\omega \psi(t)x = \ee^{t \omega} (\omega x + \xi) - \xi,
\end{displaymath}
and we have the decomposition $\xi = \omega \xi_0 +\xi_1$
where $\xi_1$ is the kernel of $\omega$
from Lemma \ref{lm:LieAlg_so(n)},
we have (\ref{eq:se_SE}) by 
by substituting $\xi$ into $\psi(t)x$.
\end{proof}

\begin{lemma} 
The map $\exp$ from $\se(n)$ to $\SE(n)$ is surjective.
\end{lemma}

\begin{proof}
Though it can be directly proved 
from the Lemma \ref{lm:LieAlg,lm:LieAlg_so(n)},
it is also obtained from (\ref{eq:se_SE}) by setting
$t=1$. 
Then its surjectivity is can be directly
proved by the triangulation of the matrix $\omega$
though it is a little bit complicate due to
the maximal tori in $\SO(n)$.
\end{proof}

\begin{corollary} \label{cor:ERT}
For every pair of congruent $n$-subspaces $M$ and $M'$,
there is a simple euclidean move $\psi(t) = \exp(t \fg)$ of 
$\fg\in\se(n)$ such that $\Phi_{\psi(1)}(M) = M'$.
\end{corollary}

\begin{definition} \label{def:PWEM}
The euclidean move $\psi(t) \in \SE(n)$ which
is given by a collection of the simple
euclidean moves 
$\{\exp(t\fg_i) \ |\ \fg_i\in \se(n), t\in [0,1]\}_{i=1,\ldots, \ell}$,
 i.e.,
\begin{displaymath}
\psi(t) = 
\exp\left(
  \frac{t-t_{j-1}}{t_j-t_{j-1}} \fg_j\right)
 \ee^{\fg_{j-1}} \cdots
 \ee^{\fg_2}
 \ee^{\fg_1}
\end{displaymath}
for $t \in 
\left[t_{j-1},t_j\right],
$
where $\displaystyle{t_j:=\frac{j}{\ell}}$,
we call $\psi(t)$ {\rm{the piecewise euclidean move}}.
\end{definition}

\section{Mathematics of Caging}

Let us consider the subspace $\cK \subset \EE^n$ whose codimension is two;
$\cK$ is $(n-2)$-subspace in $\EE^n$.
$\cK$ might be decomposed to $p$ connected parts,
\begin{displaymath}
\cK =\coprod_{i=1}^p \cK_i.
\end{displaymath}
Hereafter $\cK$ is fixed.

\begin{remark}
{\rm{
The caging is to restrict some $n$-subspace in $\EE^n$ using 
the other subspace $\cK$ whose  codimension is two, 
which is modeled after
figures, wires in three dimensional case. 
The subspace $\cK$ is modeled by them.
Thus from a practical viewpoint,
we do not consider wild geometries such as the Hilbert curve
but we neither exclude them in this article.
For such a wild object, some of the following results might
be trivial.
}}
\end{remark}

For given $\cK$, we let $[M]_{\cK^c}$ be the
subset of the configuration space $[M]$ whose element is disjoint to $\cK$,
i.e., $[M]_{\cK^c}:=\{M' \in [M]\ | \ M' \subset \cK^c\}$,
and $\fM_{\cK^c}$ be the family of $[M]_{\cK^c}$.

\begin{definition}
Let $M$ and $M'$ be congruent $n$-subspaces 
in $\cK^c$ i.e., $M$ and $M'$ belongs to $[M]_{\cK^c}
\in \fM_{\cK^c}$, and 
there exists $g\in \SE(n)$ satisfying  $\Phi_g(M)=M'$.
If there is  $\psi \in \cC^0([0,1], \SE(n))$
such that
its congruent family $\{\Phi_{\psi(t)}(M)\}_{t\in [0,1]}$ 
satisfies the conditions,

\begin{enumerate}

\item $ \Phi_{\psi(0)} = id$,

\item $ \Phi_{\psi(1)} = \Phi_g$ and

\item $\Phi_{\psi(t)}(M) \bigcap \cK = \emptyset$ for every $t\in (0,1)$,

\end{enumerate}
we say that $M$ and $M'$ are $\cK^c$-congruent,
denoted by $M \simeq_{\cK^c} M'$,
and $\Phi_{\psi}$ is a $\cK^c$-congruent homotopy
for $M$ and $M'$.
If we cannot find $\psi\in \cC^0([0,1],\SE(n))$
satisfying the above conditions, we say that
$M$ and $M'$ are not $\cK^c$-congruent,
denoted by $M \not\simeq_{\cK^c} M'$.

\end{definition}

We note that the element $[M]_{\cK^c}$ of $\fM_{\cK^c}$ is the set
of the congruent subspaces.

\begin{definition} \label{def:caging}
If every element
$M$ of $[M]_{\cK^c}$ is $\cK^c$-congruent each other, 
we say that $[M]_{\cK^c}$ is a
$\cK^c$-congruent set and
we cannot cage $[M]_{\cK^c}\in \fM_{\cK^c}$ for $\cK$.

If we find a pair $M$ and $M'$ of $[M]_{\cK^c}$ 
which are not $\cK^c$-congruent, we
call $[M]_{\cK^c}$ a complete $\cK^c$-caging set, and 
we say that we can {\rm{completely cage}} $[M]_{\cK^c}\in \fM_{\cK^c}$.
\end{definition}


\begin{proposition}
The moduli of shapes $\fM_{\cK^c}$ is decomposed to
\begin{displaymath}
\fM_{\cK^c}=\fM_{\cK^c}^{(0)} \coprod \fM_{\cK^c}^{(1)}
\end{displaymath}
where $\fM_{\cK^c}^{(0)}$ is the family of the 
$\cK^c$ congruent sets of $\fM_{\cK^c}$
and $\fM_{\cK^c}^{(1)}$ is the family of 
complete $\cK^c$-caging sets.
\end{proposition}

\begin{proof}
The decomposition is obvious from the definition.
\end{proof}

\begin{remark}
{\rm{
If $\cK$ of $n=3$ case is a space-filling curve like 
the Hilbert curve such that it is dense in $\EE^n$, 
$\fM_{\cK^c}$ itself is the empty set
though we are not concerned with such a case.
However it should be noted that 
if $\fM_{\cK^c}$ is not empty, $[M]_{\cK^c} \in \fM_{\cK^c}$
has a non-trivial geometrical structure generally. 
In fact,
Fruchard and Zamfirescu considered the similar problem in which 
the $\cK$ is a circle and $[M]_{\cK^c} (\subset [M])$
 is of a convex object
\cite{F,Z},
though in general $[M]$ 
is not convex  from a practical point of view.
}}
\end{remark}

Now in order to find a path from $M$ to $M'$, 
we express the euclidean move $\psi$ in terms of
the piecewise euclidean move in
Definition \ref{def:PWEM}.

\begin{definition} \label{def:ell-redu}
Let $M$ and $M'$ be $\cK^c$-congruent such that
$\Phi_g(M)=M'$ of $g\in \SE(n)$.
If $g$ is decomposed to
the piecewise euclidean move, i.e.,
\begin{displaymath}
g  = g_\ell \circ g_{\ell-1}
\circ \cdots
\circ g_2 \circ g_1
\end{displaymath}
satisfying the conditions;

\begin{enumerate}

\item
For each $g_i$, we have $\fg_i \in \se(n)$
satisfying $g_i = \exp(\fg_i)$ and

\item each $\psi_i \in \cC^0([0,1], \SE(n))$ given by
$\psi_i(t)=\exp(t \fg_i)$ for $t\in[0,1]$ is 
$\cK^c$-congruent homotopy for 
$\Phi_{g_{i-1} \circ \cdots \circ g_2 \circ g_1}(M)$
and
$\Phi_{g_{i} \circ \cdots \circ g_2 \circ g_1}(M)$,

\end{enumerate}
we say that
$\Phi_{g}$ is reduced to $\ell$-$\SE(n)$ action.

If $\Phi_{g}$ is reduced to $\ell$-$\SE(n)$ action but
cannot be reduced to $(\ell-1)$-$\SE(n)$ action,
we say that $\Phi_{g}$ is $\ell$-th type
and $M$ and $M'$ are $\ell$-th $\cK^c$-congruent
if $\Phi_g$ is $\ell$-th type.

Further if 
$M$ and $M'$ are congruent but not $\cK^c$-congruent,
or cannot be $\ell$-th $\cK^c$-congruent of finite $\ell$,
we say that
$M$ and $M'$ are $\infty$-th $\cK^c$-congruent.
\end{definition}

We should note that these components $\psi_i$'s are
given as the rotations at certain points or 
the translations
from Lemma \ref{lmm:RT}.

\begin{proposition}
For given an $n$-subspace $M \subset \cK^c$,
the configuration space $[M]_{\cK^c}$ of $\fM_{\cK^c}$
is decomposed to
\begin{displaymath}
[M]_{\cK^c} = \coprod_{\ell=1}^\infty [M]_{\cK^c}^{(\ell)}
\end{displaymath}
where
$[M]_{\cK^c}^{(\ell)}$ is the set of subspaces 
of the $\ell$-th $\cK^c$-congruence to $M$.
\end{proposition}

We, now, state the main theorem,
which should be contrast to  
Corollary \ref{cor:ERT}:
\begin{theorem} \label{thm:fst}
In general $\cK^c$-congruent $n$-subspaces $M$ and $M'$ are not
the first $\cK^c$-congruent.
\end{theorem}

\begin{proof}
An example is illustrated in Figure 1 of $n=2$ case,
in which 
dots mean $\cK$ and a $N$-shaped object 
corresponds to $M$.
They are $\cK^c$-congruent to the outer one but it is obvious that
the $N$-shaped object (a) in Figure 1 
 is not  the first $\cK^c$-congruent
to (d).

\begin{figure}[ht]
 \begin{center}
\hskip 1.7cm
\includegraphics[width=7cm]{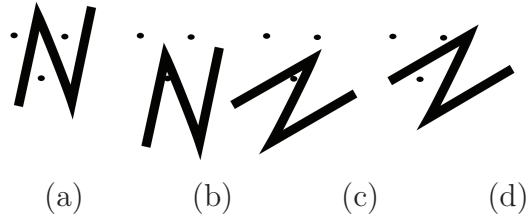}
 \newline
(a) \hskip 1.3cm
(b) \hskip 1.3cm
(c) \hskip 1.3cm
(d)
 \end{center}
\caption{
An example: The
three dots mean $\cK$ and a $N$-shaped object 
corresponds to $M$. (a), (b), (c), and (d)  
show the piecewise euclidean moves of the $N$-shaped object.
}
\label{fig:Hp}
\end{figure}

This example can be extended to $M\times \RR^{n-2}$ and
$\cK\times \RR^{n-2}$.
\end{proof}

\begin{remark}
{\rm{
The decomposition to
$\fM_{\cK^c}^{(0)}$ and $\fM_{\cK^c}^{(1)}$ is related to 
the link problem and the knot theory \cite{Ka}.
Thus $\cK^c$-congruence is a profound problem.
Some of kinds of caging are surely based on the braid group in the 
knot theory when the fundamental group 
$\pi_{1}(M)$ of $[M]$ is not trivial.

However from the practical viewpoint as mentioned in 
Introduction,
we are concerned with the $\ell$-th $\cK^c$-congruence 
rather than the $\cK^c$-congruence itself.
Theorem \ref{thm:fst} means that for 
a $\cK^c$-congruent configuration $[M]_{\cK^c}$, there might be
$M\in [M]_{\cK^c}$ so that we cannot move $M$ to $M' \in [M]_{\cK^c}$
by the simple euclidean move.
It reminds us of parking in a small garage and 
the wire puzzle; 
It is also closely related to the probabilistic 
treatment of connected path in $[M]_{\cK^c}$ \cite{DS,MP}.

Let us consider $M$ and $M'$ are
 $\ell$-th $\cK^c$-congruent, $M \simeq_{\cK^c}M'$.
If $\ell$ is not small, it is not easy to find the $\cK^c$-congruent
homotopy $\Phi_{\psi}$ such that 
$\Phi_{\psi(1)}(M) = M'$.

In the some situations, the $i$-th simple euclidean move in a
piecewise euclidean move is restricted, and thus,
countable in a certain sense, but $\ell$ is not small. 
To find the piecewise euclidean moves connecting 
$M$ and $M'$ is
basically difficult because of the combinatorial explosion.
It is the origin of the difficulty of the wire puzzle.
In other words, caging problem is 
connected with the difficulty of combinatorial problem.

This fact implies that if we want to restrict some geometrical objects
$[M]\in \fM$
by using the figures or something in a daily life, 
we do not need a complete caging in Definition \ref{def:caging}
but we have to find whether it is not small $\ell$ of
$\ell$-th $\cK^c$-congruence.

In other words, roughly speaking,
there is the degree of the difficulty to take $M$ in 
the inside of $\cK$ to the outside of $\cK$, though the inside and the outside
are not rigorous mathematically. 
We need not discriminate between the higher type of $\cK^c$-congruence and
a complete $\cK^c$-caging set in a daily life. 

Though it is very difficult to introduce the measure of
the path space in general, we could measure the difficulty
of caging and the probabilistic treatment of connected path 
in $[M]_{\cK^c}$ \cite{DS,MP}.
}}
\end{remark}

In order to express the practical caging, we introduce
another concept.

\begin{definition}
For given positive integer $\ell_0$, 
if we find $n$-subspaces $M$ and $M'$ in $\cK^c$
such that $M$ and $M'$ are congruent but not
$\ell$-th $\cK^c$-congruent for $\ell < \ell_0$,
we say that $\cK$ dissociates $M$ and $M'$ by $\ell_0$-th caging.
\end{definition}

\begin{remark}
{\rm{
Practically, 
for a given geometrical object $[M]\in \fM$,
it is very important to find  a configuration $\cK$ which
dissociates $M$ and $M'$ by $\ell_0$-th caging.

More practically, the first caging  is much more important
than higher $\ell$-th caging
case because the concerned shape $[M]$ is not so complicate.
Even for the first caging, it is not easy to find whether
it is the first caging or not because the dimension of $\SE(n)$ is
six if $n=3$. Determination of the first caging means the
determination of topological property of $\SE(n)$.
For example, if we reduce the determination of continuous space
to finite problem by expressing concerned area
and $\cK$ in terms of the voxels for 100 points per
one-dimension, we have to deal with $100^6$ data, which is huge;
we cannot deal with it practically in this stage \cite{MOO}.
Thus in order to avoid the problems,
there are several attempts and proposals including 
the C-closure method.
The second named author has investigated the problem 
 using the C-closure concept \cite{WK}.
In another way, intuitive geometrical features such as loop shape
 \cite{PSK} and \textit{double fork and neck} \cite{VKP} help us derive
 sufficient conditions for caging constraint. 
}}
\end{remark}

\noindent
Hiroyasu Hamada\\
National Institute of Technology, Sasebo College,\\
1-1, Okishin-machi, Sasebo, Nagasaki 857-1193, Japan,\\
Institute of Mathematics for Industry, Kyushu University, \\
Motooka 744, Nishi-ku, Fukuoka 819-0395, Japan\\

\noindent
Satoshi Makita\\
National Institute of Technology, Sasebo College,\\
1-1, Okishin-machi, Sasebo, Nagasaki 857-1193, Japan,\\

\noindent
Shigeki Matsutani:\\
National Institute of Technology, Sasebo College,\\
1-1, Okishin-machi, Sasebo, Nagasaki 857-1193, Japan,\\
Institute of Mathematics for Industry, Kyushu University, \\
Motooka 744, Nishi-ku, Fukuoka 819-0395, Japan\\
smatsu@sasebo.ac.jp\\
\end{document}